\documentclass[12pt]{amsart}

\usepackage{amsmath, amssymb, amscd, verbatim, xspace,amsthm}
\usepackage{latexsym, epsfig, color, fullpage}

\newcommand{\A}{\ensuremath{{\mathbb{A}}}}
\newcommand{\C}{\ensuremath{{\mathbb{C}}}}
\newcommand{\Z}{\ensuremath{{\mathbb{Z}}}\xspace}
\renewcommand{\P}{\ensuremath{{\mathbb{P}}}}
\newcommand{\Q}{\ensuremath{{\mathbb{Q}}}}

\newcommand{\ra}{\rightarrow}
\newcommand{\lra}{\longrightarrow}
\newcommand\Det{\operatorname{Det}}

\newcommand\sHom{\operatorname{\mathcal{H}om}}

\newcommand\im{\operatorname{im}}
\newcommand\Stab{\operatorname{Stab}}

\newcommand\Sym{\operatorname{Sym}}

\newcommand\Tr{\operatorname{Tr}}

\newcommand\tensor{\otimes}
\newcommand\isom{\cong}
\newcommand\sub{\subset}
\newcommand\tesnor{\otimes}

\newcommand\disc{\operatorname{disc}}
\newcommand\Disc{\operatorname{Disc}}
\newcommand\GL{\operatorname{GL}}

\newcommand\SL{\operatorname{SL}}

\newcommand\Spec{\operatorname{Spec}}

\newcommand\Rf{\ensuremath{R_f}\xspace}

\renewcommand\O{\mathcal{O}}

\newcommand\BS{\ensuremath{S}\xspace}

\newcommand\OS{\ensuremath{{\O_\BS}}\xspace}

\newcommand\map[4]{\ensuremath{\begin{array}{ccc}#1&\lra&#2\\#3&\mapsto&#4\end{array}}}
\newcommand\bij[2]{\ensuremath{\left\{\parbox{2.5 in}{#1}\right\} \longleftrightarrow \left\{\parbox{2.5 in}{#2}\right\}}}
\newcommand\mapin[2]{\ensuremath{\left\{\parbox{2.5 in}{#1}\right\} \longrightarrow \left\{\parbox{2.5 in}{#2}\right\}}}

\newcommand\bq{\begin{equation}}
\newcommand\eq{\end{equation}}

\newtheorem{proposition}{Proposition}[section]
\newtheorem{theorem}[proposition]{Theorem}
\newtheorem{corollary}[proposition]{Corollary}

\newtheorem{lemma}[proposition]{Lemma}

\theoremstyle{remark}

\newtheoremstyle{citing}
 {3pt}
 {3pt}
 {\itshape}
 {}
 {\bfseries}
 {.}
 {.5em}
 {\thmnote{#3}}

\theoremstyle{citing}
\newtheorem*{varthm}{}

\newenvironment{definition}{\vspace{2 ex}{\noindent{\bf Definition. }}}{\vspace{2 ex}}

\begin{document}

\title{Quartic rings associated to binary quartic forms}

\author{Melanie Matchett Wood}
\thanks{email: mwood@math.stanford.edu}

\address{%
American Institute of Mathematics and Stanford University\\
Department of Mathematics\\
Building 380, Sloan Hall\\
Stanford, California 94305\\
   USA
}



\begin{abstract}
We give a bijection between binary quartic forms and quartic rings with
a monogenic cubic resolvent ring, relating the rings associated to binary quartic forms with Bhargava's cubic resolvent rings.  This gives a parametrization of quartic rings with monogenic cubic resolvents. 
We also give a geometric interpretation of this parametrization.
\end{abstract}

\maketitle

\section{Introduction}
Algebraic objects associated to binary forms have long been studied.  Dedekind originally associated a quadratic ring and ideal class to every binary quadratic form \cite{DD}.  In fact, binary quadratic forms exactly parametrize ideal classes of quadratic rings (see \cite[Section 5.2]{Cohen}, \cite{Kneser}, or  \cite{BinQuad} for a treatment that includes all binary quadratic forms, even the zero form!).
In 1940, Delone and Faddeev \cite{DF} associated cubic rings to binary cubic forms and found that binary cubic forms exactly parametrize cubic rings  (see also \cite{GGS} for a treatment of all binary cubic forms).

In fact, one can associate an $n$-ic ring (a ring isomorphic to $\Z^n$ as a $\Z$ module) to a binary $n$-ic form for any $n$.
(When $n=2,3,4$, we also call an $n$-ic ring \emph{quadratic, cubic, quartic}, respectively.)  
Early work on the rings associated to binary forms was done by Birch and Merriman \cite{BM} and Nakagawa \cite{Naka}.  In \cite{PrimeSplit}, 
Del Corso, Dvornicich, and Simon determine the splitting of the prime $p$ in such a ring in terms of the factorization of the binary $n$-ic form modulo $p^k$.
In \cite{SimonIdeal}, Simon associates an ideal class of the associated ring to a binary $n$-ic form, and in \cite{Simonobst} he applies this ideal class to study integer solutions to equations of the form $Cy^d=F(x,z)$, where $F$ is a binary form. 
In \cite{binarynic}, it is determined exactly what algebraic structures are parametrized by binary $n$-ic forms, for all $n$.  This structure is a rank $n$ ring and an ideal class for that ring, such that the 
action of the ring on the ideal class satisfies a certain exact sequence (which comes naturally from geometry).  When $n=2$, the exact sequence condition is vacuous, and when $n=3$ 
the condition forces the ideal class to be the unit ideal.  In this note we give a different point of view (from \cite{binarynic}) on the algebraic data parametrized by binary quartic forms.
We prove the following, which is the main result of this paper.
\begin{theorem}\label{T:Main}
There is a natural, discriminant preserving bijection between the set of $\GL_2(\Z)$-equivalence classes of binary quartic forms and the set of isomorphism classes of pairs $(Q,C)$ where $Q$ is a quartic
ring and $C$ is a monogenic cubic resolvent of $Q$ (where isomorphisms are required to preserve the generator of $C$ modulo $\Z$).
\end{theorem}

In particular, in this paper we construct the bijection of Theorem~\ref{T:Main} explicitly.  For example, the construction of the quartic ring is given in Equation~\eqref{E:multtable}.
 The $\GL_2(\Z)$ action on binary quartic forms is given by 
$\left(\begin{smallmatrix}
                         a & b\\
			c & d
                        \end{smallmatrix}\right)\circ f(x,y)= f(ax+cy,bx+dy)$.
A \emph{monogenic} ring is one that is generated by one element as a $\Z$-algebra. 
  The above simple criteria for when a quartic ring is associated to a binary quartic form is an application of the notion of the \emph{cubic resolvent} of a quartic ring, which was introduced by Bhargava \cite{HCL3} in his parametrization of quartic rings with their cubic resolvents by pairs of ternary quadratic forms. 
Theorem \ref{T:Main} is used by Bhargava and Shankar \cite{BS} in their determination of the average number of $2$-torsion elements in the class groups in monogenic maximal cubic orders.  Surprisingly, these averages are different than for general maximal cubic orders!

We recall Bhargava's parametrization of quartic rings and their cubic resolvents here.  If we write ternary quadratic forms as matrices, we can give the $\GL_3(\Z)$ action on pairs of ternary quadratic forms as $(A,B)\mapsto (gAg^t, gBg^t)$ for $g\in\GL_3(\Z)$.
The action of $g=\left(\begin{smallmatrix}
                         a & b\\
			c & d
                        \end{smallmatrix}\right)\in\GL_2(\Z)$ on a pair $(A,B)$ of ternary quadratic forms takes
$(A,B)$ to $(aA+bB,cA+dB)$.
\begin{theorem}\cite[Theorem 1]{HCL3}\label{T:B}
 There is a natural, discriminant preserving bijection between the set of $\GL_3(\Z) \times \GL_2(\Z)$-equivalence classes of pairs of ternary quadratic forms and isomorphism classes $(Q,C)$ where $Q$ is a quartic ring and $C$ is a cubic resolvent of $Q$.
\end{theorem}

 Cubic resolvents are integral models of the classical notion of the cubic resolvent field of a quartic field, and they have the same discriminant as their quartic ring.  
A cubic resolvent $C$ of a quartic ring $Q$ comes with a quadratic map from $Q$ to $C$ (suppressed from the notation $(Q,C)$), i.e. a function $q: Q\ra C$ such that for $a\in \Z$, we have $q(ax)=a^2q(x)$, and also such that $q(x+y)-q(x)-q(y)$ is a bilinear form in $x$ and $y$.
When $Q$ is a quartic order in a field $\Q(\alpha)$ whose Galois closure has Galois group $S_4$ or $A_4$, then $\Q(\alpha\alpha'+\alpha'\alpha'')$ is the classical cubic resolvent field of $\Q(\alpha)$, where
$\alpha',\alpha'',\alpha'''$ denote the conjugates of $\alpha$.
In this case, a cubic resolvent ring of $Q$ can be defined to be an order $C$ in the cubic resolvent field with the same discriminant as $Q$ such that for all $x\in Q$, we have $xx'+x''x'''\in C$.  The quadratic map from $Q$ to $C$ is $x\ra xx'+x''x'''$.
For other quartic rings $Q$, the definition is a little more subtle (in particular there is no classical cubic resolvent field) and is given in the Appendix Section~\ref{A:CubRes}. 
 Every quartic ring has a cubic resolvent ring and quartic maximal  orders have a unique cubic resolvent ring \cite[Corollary 4]{HCL3}.
An isomorphism $(Q,C)$ to $(Q',C')$ of pairs each consisting of a quartic ring and a cubic resolvent of that quartic ring is given by 
ring isomorphisms $Q\isom Q'$ and $C \isom C'$ that commute with the quadratic maps $Q\ra C$ and $Q' \ra C'$.  Note that in Theorem~\ref{T:Main} the pairs $(Q,C)$ also come with a generator of $C$ as a $\Z$-algebra, and in an isomorphism between $(Q,C)$ and $(Q',C')$, the induced map $C/\Z \isom C'/\Z$ must take the chosen generator of $C$ to the chosen generator of $C'$.

Theorem~\ref{T:B} gives a parametrization of all quartic rings with their cubic resolvents.  Most cubic rings are not generated by one element as a $\Z$-algebra, but the special cubic rings which are generated by one element are called monogenic.  Our Theorem~\ref{T:Main} gives a parametrization of quartic rings with monogenic cubic resolvents.  Alternatively, Theorem~\ref{T:Main} can be viewed as answering the question ``In the association of quartic rings to binary quartic forms (as in \cite{Naka, PrimeSplit, SimonIdeal, binarynic}), which quartic rings appear?'')

We prove Theorem~\ref{T:Main} by mapping binary quartic forms to pairs of ternary quadratic forms in a way the respects the constructions on the associated quartic rings (as in 
\cite{Naka, PrimeSplit, SimonIdeal,binarynic} and \cite{HCL3} respectively). 
This map is given in Section~\ref{S:constructions}.
 We then see that the rings associated to binary quartic forms have monogenic cubic resolvents, and then that any quartic ring with a monogenic cubic resolvent is associated to some binary quartic form.  We then in Section~\ref{S:GLaction} study how the $\GL_2(\Z)$-action on binary quartic forms changes the associated quartic ring, which allows us to prove our main result (Theorem~\ref{T:Main}) in Section~\ref{S:Main}, after recording some preliminaries about monogenized cubic rings in Section~\ref{S:cubic}.  
In Section~\ref{S:Geom}, we explain the results of this paper from a geometric point of view, and discuss analogs in which the integers are replaced by an arbitrary scheme.  In Section~\ref{S:Canonical}, we see how the $\GL_2(\Z)$ invariants of a binary quartic form are related to the monogenized cubic resolvent ring of the associated quartic ring.

\section{Constructions}\label{S:constructions}
In this section, we give the constructions of rings from forms mentioned in the introduction, as well as the relationship between binary quartic forms and pairs of ternary quadratic forms that will be
the basis of the proof of Theorem~\ref{T:Main}.  A \emph {based $n$-ic ring} is an $n$-ic ring $R$, along with a choice of 
$\zeta_1,\dots,\zeta_{n-1}\in R/\Z$ such that $1,\zeta_1,\dots,\zeta_{n-1}$ is a $\Z$-module basis of $R$ (which clearly does not depend on the lift of $\zeta_i$ to $R$).
If $R$ is an $n$-ic ring, then for $r\in R$, multiplication by $r$ is a linear transformation on the $\Z$-module $R$, and we write $\Tr(r)$ for the trace of that linear transformation.
The discriminant of a $n$-ic ring with $\Z$-module basis $\zeta_i$ is the determinant of the matrix with $i,j$ entry $\Tr(\zeta_i\zeta_j)$ (which is easily seen to not depend on the choice of basis).

\subsection{Construction of a ring from a binary form}\label{S:conbf}
Given a binary $n$-ic form, 
$
f=f_0x^n +f_1 x^{n-1}y +\dots +f_ny^n \quad \textrm{with}\quad f_i\in\Z,
$
such that $f_0\ne 0$, we can form a based $n$-ic ring
$\Rf$ as a subring of $\Q(\beta)/(f_0\beta^n +f_1 \beta^{n-1} +\dots +f_n)$ 
with $\Z$-module basis
 \begin{gather}\label{E:basis}
 \zeta_0=1\\\notag
 \zeta_1=f_0\beta\\\notag
 \zeta_2=f_0\beta^2+f_1\beta\\\notag
 \vdots\\\notag
 \zeta_k=f_0\beta^k+\dots+f_{k-1}\beta\\\notag
 \vdots\\\notag
 \zeta_{n-1}=f_0\beta^{n-1}+\dots+f_{n-2}\beta,
  \end{gather}
as first considered by Birch and Merriman \cite{BM}.
It is shown that $\Rf$ is a ring in \cite[Proposition 1.1]{Naka}.
 We have the discriminant equality 
$\Disc\Rf=\Disc f$
(see, for example, \cite[Proposition 4]{Simon}).
In Section~\ref{S:Geom}, we will recall a less explicit, but more natural geometric construction of $R_f$ that was developed in \cite{binarynic}.
In this paper we use this construction when $n=3,4$.

Given a binary quartic form, $f= f_0 x^4 +f_1 x^3y+f_2x^2y^2+f_3 xy^3 +f_4y^4 $ with $f_i\in\Z$ and $f_0\ne 0$, we can
work out the multiplication table of $\Rf$ explicitly as follows, letting $\zeta_3'=\zeta_3+f_3$:
\begin{align}\label{E:multtable}
\zeta_1^2 &=& -f_1 \zeta_1& + f_0\zeta_2& & & \notag\\
\zeta_1\zeta_2 &=& -f_2 \zeta_1& & +f_0\zeta_3'& -f_0f_3& \notag\\
\zeta_1\zeta_3' &=& & & & -f_0f_4& \\
\zeta_2\zeta_2 &=& -f_3 \zeta_1& -f_2\zeta_2& +f_1\zeta_3'& -f_1f_3-f_0f_4& \notag\\
\zeta_2\zeta_3' &=& -f_4 \zeta_1& & & -f_1f_4&\notag\\
(\zeta_3')^2&=& & -f_4\zeta_2& +f_3\zeta_3'& -f_3^2-f_2f_4&\notag. 
\end{align}            
Even if $f_0=0$, we can use the above multiplication table to construct a based quartic ring $R_f$ from a binary quartic form $f$.  For example, when $f=0$, we see that $R_f=\Z[\zeta_1,\zeta_2,\zeta_3']/(\zeta_1,\zeta_2,\zeta_3')^2$.

Given a binary cubic form, $f= a x^3 +b x^2y+cxy^2+d y^3$ with $a,b,c,d\in\Z$, we let $\omega=-\zeta_1$ and $\theta=-\zeta_2-c$ and have the multiplication table for $\Rf$ as follows:
\begin{align}
 \omega\theta&= -ad \notag\\ 
\omega^2 &= -ac +b\omega -a \theta\\
\theta^2&= -bd+d\omega-c\theta. \notag
\end{align}
Given a based cubic ring $C$ with a $\Z$-module basis $\omega,\theta$ for $C/\Z$, there is a unique choice of lifts of $\omega,\theta$ to $C$ such that $\omega\theta\in \Z$.
Thus a based cubic ring is the same as a cubic ring with choice of $\Z$-module basis $1,\omega,\theta$ such that $\omega\theta\in \Z$.
The construction of $R_f$ from $f$ is equivariant under a $\GL_2(\Z)$ action that we specify here. 
Let $g=\left(\begin{smallmatrix}
                         a & b\\
			c & d
                        \end{smallmatrix}\right)$ be an element of $\GL_2(\Z)$ and $f=F(x,y)$ be a binary cubic form.
Then $g\circ f= \frac{1}{ad-bc}F(ax+cy,bx+dy)$.  If $\omega,\theta$ is a basis of $C/\Z$, then after action by $g$, the new basis of $C/\Z$ is $\omega',\theta'$, where 
$\left[\begin{smallmatrix}
                         \omega'\\
			\theta'
                        \end{smallmatrix}\right]=g\left[\begin{smallmatrix}
                         \omega\\
			\theta
                        \end{smallmatrix}\right]$.  We can now recall the parametrization of cubic rings by binary cubic forms, discovered first by Delone and Faddeev \cite{DF} and given more recently and in language closer to ours by Gan, Gross, and Savin \cite[Proposition 4.2]{GGS}.

\begin{theorem}[c.f. \cite{GGS}]\label{T:cubic}
 The above construction gives a bijection between based cubic rings and binary cubic forms.  
The bijection between based cubic rings and binary cubic forms is equivariant for the above $\GL_2(\Z)$ actions, giving a bijection between cubic rings and $\GL_2(\Z)$-classes of binary cubic forms.
\end{theorem}

\subsection{Construction of a quartic ring and cubic resolvent from a pair of ternary quadratic forms}
We can represent a pair of ternary quadratic forms by a pair of matrices $(A,B)$
such that
$$
A=\begin{pmatrix} a_{11} & \frac{a_{12}}{2} & \frac{a_{13}}{2} \\
\frac{a_{12}}{2} & a_{22} & \frac{a_{23}}{2} \\
\frac{a_{13}}{2} & \frac{a_{23}}{2} & a_{33} 
\end{pmatrix}
\qquad 
B=\begin{pmatrix} b_{11} & \frac{b_{12}}{2} & \frac{b_{13}}{2} \\
\frac{b_{12}}{2} & b_{22} & \frac{b_{23}}{2} \\
\frac{b_{13}}{2} & \frac{b_{23}}{2} & b_{33} 
\end{pmatrix}
$$
with $a_{ij},b_{ij}\in \Z$.
Bhargava, in \cite[Section 3.2]{HCL3}, constructed a based quartic ring $Q$ from $(A,B)$, by giving the multiplication table explicitly in terms of the $a_i$ and $b_i$.  See also
Section~\ref{S:Geom} for a geometric description of this quartic ring that was developed in \cite{Quartic}.
We define the \emph{determinant} of the pair $(A,B)$ to be the binary cubic form $4\Det(Ax-By)$.
The based cubic resolvent associated to a pair $(A,B)$ is given by this determinant binary cubic form via Theorem~\ref{T:cubic}.

We have a $\GL_3(\Z)$ action on pairs of ternary quadratic forms given by $(A,B)\mapsto (gAg^t, gBg^t)$ for $g\in\GL_3(\Z)$.
The action of $g=\left(\begin{smallmatrix}
                         a & b\\
			c & d
                        \end{smallmatrix}\right)\in\GL_2(\Z)$ on a pair $(A,B)$ of ternary quadratic forms takes
$(A,B)$ to $(aA+bB,cA+dB)$.
In Bhargava's construction of a quartic ring $Q$ and cubic resolvent from a pair of ternary quadratic forms, the $\GL_2(\Z)\times\GL_3(\Z)$ action on the pair of ternary quadratic forms corresponds to changing the basis of $Q$ and its cubic resolvent, and thus giving the map in Theorem~\ref{T:B} (see \cite{HCL3}).

\subsection{Construction of a pair of ternary quadratic forms from a binary quartic form}
The new construction of this paper is the map
$$
\Psi : \mapin{binary quartic forms}{pairs of ternary quadratic forms}.
$$
which sends $f=f_0 x^4 +f_1 x^3y+f_2x^2y^2+f_3 xy^3 +f_4y^4$ to $(A_0,B_f)$, where
$$
A_0=\begin{pmatrix} 0 & \frac{-1}{2} & 0 \\
\frac{-1}{2} & 0 & 0 \\
0 & 0 & 1 
\end{pmatrix}\quad\text{and}\quad
B_f=\begin{pmatrix} f_4 & 0 & \frac{f_3}{2} \\
0 & f_0 & \frac{f_1}{2} \\
\frac{f_3}{2} & \frac{f_1}{2} &  f_2
\end{pmatrix}
$$
See Section~\ref{SS:geomZ} for a geometric explanation of why we these particular forms arise.
One then naturally puts an equivalence on pairs of ternary quadratic forms
such that $(A,B)\sim (A, B +nA)$ for $n\in \Z$, and we can also consider
$$
\bar{\Psi} : \mapin{binary quartic forms}{pairs of ternary quadratic forms$/\sim$}.
$$
which sends $f$ to $(A_0,B_f + \Z A_0)$.
Note that $\bar{\Psi}$ is injective and its image is all classes of pairs $(A_0,B)$.

We have seen above that the binary quartic form $f$ gives a based quartic
ring $R_f$ over \Z, and a pair of ternary quadratic forms
gives a based quartic ring $Q$ and a based cubic resolvent $C$ for that quartic ring.  
The map $\Psi$ may look straightforward, or even arbitrary, but it has been carefully constructed so that $\Psi$
(and also $\bar{\Psi}$) respects the above constructions of based quartic rings and so as to satisfy Theorem~\ref{T:action}, which says that the map is $\GL_2(\Z)$ equivariant.

\begin{lemma}\label{L:multable}
 If a based quartic ring $Q$ is associated to the pair $\Psi(f)=(A_0,B_f)$ or any element of the class
$\bar{\Psi}(f)=(A_0,B_f + \Z A_0)$, then $R_f=Q$.  
\end{lemma}

\begin{proof}
 We have a basis $\zeta_1,\zeta_2,\zeta_3$  of the quartic ring associated to a binary quartic form as given in Section~\ref{S:conbf}.  
We let $\zeta_3'=\zeta_3+f_3$. 
We let $\alpha_1=\zeta_3'$ and $\alpha_2=\zeta_1$ and $\alpha_3=\zeta_2$.
We then see from Equation~\eqref{E:multtable} that the $\alpha_i$ satisfy the multiplication table given in \cite[Equations (21) and (23)]{HCL3} for the pair $(A_0,B_f)$.
\end{proof}

Note that
when we have based rings, it makes
sense to talk about equality and not just isomorphism. 
All of the elements in the class $(A_0,B_f + \Z A_0)$
give the same based quartic ring and the same cubic resolvent, but with different bases
for the cubic resolvent.

Since $4\det(A_0)=-1$, any element of $\bar{\Psi}(f)$ has a based cubic ring
given by a cubic form with coefficient $-1$ of $x^3$.  In particular, the cubic resolvent ring $C$ associated to $\Psi(f)$
has $\Z$-module basis $1,\omega,\theta$ with $\omega^2 = -c + b\omega + \theta$
with $b,c\in\Z$, and thus $1,\omega,\omega^2$ is a $\Z$-module basis
of $C$.  
A \emph{monogenized cubic ring} is a cubic ring $C$ and an element 
$\omega\in C/\Z$ such that $C=\Z[\omega]$.  
An isomorphism of monogenized cubic rings must preserve the chosen element of $C/\Z$.
A \emph{monogenized based cubic ring} is a based cubic ring $C$
with basis $1,\omega,\theta$,
such that $1,\omega,\omega^2$ is a $\Z$-module basis for the ring
\emph{of the same orientation as $1,\omega,\theta$} (in other words, such that $\theta \in \omega^2 + \omega\Z +\Z$), or equivalently a based cubic ring $C$
that corresponds to a binary cubic form with
$x^3$ coefficient $-1$.

\begin{proposition}
 Any element of $\bar{\Psi}(f)$ corresponds to a quartic ring with a monogenized based cubic
 resolvent.
\end{proposition}

\begin{corollary}
 Any ring $R_f$ from a binary quartic form has a monogenic cubic resolvent.
\end{corollary}

We will see that the converse is also true.

\begin{theorem}\label{T:surjective}
If a quartic ring $Q$ has a monogenic resolvent $C$, then there
exist bases of $Q$ and $C$ such that the based pair $(Q,C)$ corresponds
to $(A_0,B)$ in the parametrization of quartic rings and cubic resolvents.
\end{theorem}
\begin{proof}
Recall that $g\in \SL_3(\Z)$ acts on $A$ by sending it to
$g A g^t$. 
There is only one $\SL_3(\Z)$ class of ternary
quadratic forms with determinant $-1/4$.
This is a classical fact from the theory of ternary quadratic forms, see e.g. \cite[Chapter 5, Theorem 6.3]{Fla}.
   Then we can conclude that all 
such forms are in the $\SL_3(\Z)$ class of $A_0$.
If we have a pair $(A,B)$ corresponding to a quartic ring $Q$
with a monogenic cubic resolvent $C$, we can choose a monogenized basis of $C$
 so that we can assume $\Det(A)=-1/4$.
Then we can act by an element $g\in \SL_3(\Z)$
so that we obtain $A=A_0$.
\end{proof}

\begin{corollary}
 All quartic rings with monogenic resolvents are the ring $R_f$ constructed from some
binary quartic form $f$.
\end{corollary}

In Section~\ref{S:GLaction} we see how the $\GL_2(\Z)$ action on binary quartic forms interacts with the construction $\Psi$.

\section{$\GL$ action on forms}\label{S:GLaction}
There is a natural (left) $\GL_2(\Z)$ action on binary quartic forms.  Let  $g=\left(\begin{smallmatrix}
                         a & b\\
			c & d
                        \end{smallmatrix}\right)$ be an element of $\GL_2(\Z)$ and $f=F(x,y)$ be a binary quartic form.
Then $g\circ f= F(ax+cy,bx+dy)$.  Note that this action has a kernel of $\pm 1$.
Recall that the $\GL_3(\Z)$ action on pairs of ternary quadratic forms is given by $(A,B)\mapsto (gAg^t, gBg^t)$ for $g\in\GL_3(\Z)$.

\begin{theorem}\label{T:action}
 The map 
$$
\map{\rho: \GL_2 (\Z)}{\SL_3 (\Z)}{\left(\begin{matrix}
                         a & b\\
			c & d
                        \end{matrix}\right)
}{\frac{1}{ad-bc}\left(\begin{matrix}
                         d^2 & c^2 & dc\\
			b^2 & a^2 &  ab \\
			2bd & 2ac & ad+bc
                        \end{matrix}\right)
}
$$
is a homomorphism, and gives  a $\GL_2(\Z)$ action on pairs of ternary quadratic forms for which $\bar{\Psi}$ is equivariant, i.e.
for $g\in\GL_2(\Z)$ we have $\bar{\Psi}(g\circ f)=g\bar{\Psi}(f)$. 
We have $\im(\rho)\subset \Stab(A_0)$. 
\end{theorem}
\begin{proof}
 It is easy to compute that $\rho$ is a homomorphism, and it can also can be realized as the representation of $\GL_2(\Z)$ on
binary quadratic forms (up to a twist by the determinant).  We can check the equivariance of $\bar{\Psi}$ 
 by computation (which simplifies on generators 
$\left(\begin{smallmatrix}
                         0 & 1\\
			1 & 0
                        \end{smallmatrix}\right)$,
$\left(\begin{smallmatrix}
                         1 & 0\\
			0 & -1
                        \end{smallmatrix}\right)$, and
$\left(\begin{smallmatrix}
                         1 & 1\\
			0 & 1
                        \end{smallmatrix}\right)$
of $\GL_2(\Z)$).
Let $$Y=
\left(\begin{matrix}
                         d^2 & c^2 & dc\\
			b^2 & a^2 &  ab \\
			2bd & 2ac & ad+bc
                        \end{matrix}\right),
$$ and $Y'=\frac{1}{ad-bc}Y$.
We can compute formally that $Y' A_0 (Y')^t=A_0$.  
We can also compute formally that $Y$ gives the correct action on $B+A_0\Z$ exactly;
if $\Psi(f)=(A_0,B)$ and $\Psi(g\circ f)=(A_0,B')$, then
$Y (B+A_0\Z)Y^t=B'+A_0\Z$.  
Since $ad-bc=\pm 1$, we have that $Y (B+A_0\Z)Y^t=Y' (B+A_0\Z)(Y')^t$.
\end{proof}

The following Lemma is crucial to our main theorem, and is proven in \cite[Chapter 13, Lemma 5.2]{Cassels}.  
\begin{lemma}\label{L:Stab}
 We have $\im(\rho)=\Stab(A_0)$.
\end{lemma}

\section{Monogenized Cubic Rings}\label{S:cubic}

Note that given a monogenized cubic ring, $C,\omega$, there is a unique choice of $\theta$ in
$C/(\Z\oplus \omega\Z)$ that lifts to a monogenized basis of $C$ because of the orientation requirement.
We define $N$ to be the subgroup $\left(\begin{smallmatrix}
                         1 & 0\\
			* & 1
                        \end{smallmatrix}\right)$
of $\GL_2(\Z)$.  
The action of $N$  on binary cubic forms fixes their $x^3$ coefficient.  Moreover, $N$ acts on the basis of $C/\Z$ of a based cubic ring $C$ and
fixes the first basis element.  We also have that $N$ acts on pairs of ternary quadratic forms, and fixes the first form in the pair.

\begin{proposition}
We have that $N$ classes of binary cubic forms with $x^3$ coefficient $-1$ are in bijection with isomorphism classes of monogenized cubic rings. 
\end{proposition}
\begin{proof}
We have that binary cubic forms with $x^3$ coefficient $-1$ are in bijection with based cubic rings in which
$1,\omega,\omega^2$ is a basis of the same orientation as the given basis $1,\omega,\theta$.  
When we pass to  $N$ classes of forms, the  correspondence is 
to cubic rings with a choice of $\omega\in C/\Z$ and $\theta \in C/ (\Z\oplus\omega\Z)$
such that $1,\omega,\omega^2$ is a basis of the same orientation as $1,\omega,\theta$.
However, given $\omega$, the only such choice of $\theta\in C/ (\Z\oplus\omega\Z)$ is $\theta=\omega^2$.
\end{proof}

\section{Main Theorem}\label{S:Main}

In this section, we prove the main theorem of this paper.  
\begin{varthm}[Theorem~\ref{T:Main}]
There is a bijection between the set of $\GL_2(\Z)$-equivalence classes of binary quartic forms and the set of isomorphism classes of pairs $(Q,C)$ where $Q$ is a quartic
ring and $C$ is a monogenized cubic resolvent of $Q$.
\end{varthm}
An isomorphism of a pair $(Q,C)$ where $C$ is monogenized, is just an isomorphism of the underlying pair of quartic ring and cubic resolvent such that the isomorphism between cubic rings preserves the chosen generator modulo $\Z$.  
\begin{proof}
 So far, we have established a bijection
$$
\bij{binary quartic forms}{$N$ classes of pairs $(A_0,B)$ of ternary quadratic forms},
$$
where $A_0$ is the fixed form defined in the Introduction, and $B$ is any ternary quadratic form.  From the parametrization 
of quartic rings \cite{HCL3}, we know that  $N$ classes of pairs $(A_0,B)$ of ternary quadratic forms
are in bijection with $(Q,C)$, where $Q$ is a based quartic ring, $C$ is an $N$ class of based cubic resolvent rings, and the
resolvent map is given by $(A_0,B)$.  Since $4\Det(A_0)=-1$, the $N$ class of bases of $C$ exactly corresponds to a monogenization of $C$.  Thus we have a bijection
$$
\bij{binary quartic forms}{$(Q,C)$, where $Q$ is a based quartic ring, $C$ is a monogenized cubic resolvent ring, and the
resolvent map is given by $(A_0,B)$}.
$$
We know that in this map the $\GL_2(\Z)$ action on binary quartic forms just corresponds to a $\SL_3(\Z)$ change of basis of $Q$, and thus gives the same isomorphism class of $(Q,C)$.  Thus the map from 
$\GL_2(\Z)$ classes of binary quartic forms to isomorphism classes of $(Q,C)$ is well-defined.  
We know the map is surjective by Theorem \ref{T:surjective}.
To show it is injective, suppose we have two pairs $(Q,C)$
and $(Q',C')$ of quartic rings with monogenized cubic resolvents.  We can choose bases for the quartic rings so that
the resolvent maps are given by $(A_0,B+A_0\Z)$ and $(A_0,B'+A_0\Z)$.
If we have an isomorphism of the pairs $(Q,C)$
and $(Q',C')$, it must come from an element $(g,h)\in \GL_2(\Z)\times \GL_3(\Z)$ with $\det(g)\det(h)=1$.
Since $g$ fixes $\omega$ and the orientation of the cubic ring (as both cubic forms have $x^3$ coefficient $-1$), it
must be an element of $N$.  Then $\det(h)=1$, and we see that the isomorphism comes from an element of $\SL_3(\Z)$ that fixes $A_0$.  
By Lemma~\ref{L:Stab}, such an element is in the image of the $\rho$ of Theorem~\ref{T:action}, and thus 
$(A_0,B+A_0\Z)$ and $(A_0,B'+A_0\Z)$ come from the same $\GL_2(\Z)$ class of binary quartic forms. 
\end{proof}

A quartic ring might have multiple cubic resolvents, only some of which are monogenic.
In our bijection, the quartic ring appears once for each monogenized resolvent.   If it has a cubic resolvent monogenic
in two different ways then it will appear for each of those monogenizations of the cubic ring. 
Also, note that the binary quartic form $-x^3y+bx^2y^2+cxy^3 +dy^4$ maps to $(A_0,B)$ with determinant
$-x^3+bx^2y+cxy^2 +dy^3$.  Thus every monogenized
cubic ring appears as a resolvent of some quartic ring.

\section{Geometric Interpretation}\label{S:Geom}

We now give a geometric interpretation of Theorem~\ref{T:Main}, which was proven in the last section.  This geometric interpretation relies on the geometric constructions of a ring from a binary form and from a pair of ternary quadratic forms from \cite{binarynic} and \cite{Quartic}, respectively.

\subsection{Geometric constructions}\label{SS:gc}
Given a binary $n$-ic form $f$, we have a map $$\O_{\P^1_{\Z}}(-n) \stackrel{f}{\ra} \O_{\P^1_{\Z}}$$ of sheaves on $\P^1_{\Z}$.  The image of this map defines
an ideal sheaf of $\P^1_{\Z}$, corresponding to a subscheme of $\P^1_\Z$ that we we call $S_f$, \emph{the scheme cut out by the form $f$}.  The scheme $S_f$ is the scheme of roots of the form $f$
in $\P^1_\Z$.  We then have a ring $H^0(S_f,\O_{S_f})$ of global functions of the scheme $S_f$.  
This construction of a ring from a binary form was given in the case $n=3$ by Deligne in \cite{Delcubic} and treated completely for all $n$ in \cite{binarynic}.
When $f$ has at least one non-zero coefficient, we have $R_f=H^0(S_f,\O_{S_f})$, which is proven in \cite[Theorem 2.4]{binarynic}. 
(See Section~\ref{SS:AB} or \cite[Section 2.4]{binarynic} for a geometric construction that works even when $f=0$.)
 In other words, the algebraic construction given in Section~\ref{S:constructions}
and this geometric construction of a ring from a binary form agree.  

Next we recall the geometric construction, given in \cite{Quartic}, of a ring from a pair of ternary quadratic forms.
Given a pair $(A,B)$ of ternary quadratic forms, we have a map
$$ \O_{\P^2_{\Z}}(-2)^{\oplus 2} \stackrel{\left[\begin{smallmatrix}
                                A \\
B
                               \end{smallmatrix}\right]}{\ra} \O_{\P^2_{\Z}}$$ 
of sheaves on $\P^2_{\Z}$.  The image of this map defines
an ideal sheaf of $\P^2_{\Z}$, corresponding to a subscheme of $\P^2_\Z$ that we we call $S_{(A,B)}$, \emph{the scheme cut out by the pair $(A,B)$}.  The scheme $S_f$ is the scheme of common roots of the quadratic forms $A,B$.  We have a ring $H^0(S_{(A,B)},\O_{S_{(A,B)}})$ of global functions on $S_{(A,B)}$.  When $A,B$ have no common non-unit factors in the ring $\Z[x,y,z]$, then
this ring agrees with Bhargava's \cite{HCL3} construction of a quartic ring from $(A,B)$.  This agreement in proven in \cite[Theorem 5.1]{Quartic}, which also gives a geometric construction of a ring from all pairs $(A,B)$; see also the letter \cite{Delquartic} from Deligne to Bhargava, which gives a geometric approach to constructing quartic rings from pairs of ternary quadratic forms that is different from the approach in \cite{Quartic}.  

 When $A,B$ generate a 2 dimensional subspace of ternary quadratic forms over $\Q$ and every $\Z/p$ (call this \emph{nice}), then
	it makes sense to talk about the $\P^1_\Z$ (pencil) of conics through $A$ and $B$ (see \cite[I.6.2:Example 1, II.6.4:Example 1]{Shaf} for an introduction to the idea of a pencil of conics).  Recall that the cubic resolvent ring associated to $(A,B)$ is the cubic ring
associated to the binary cubic form $4\Det(Ax-By)$.  
  A conic
given by a symmetric matrix $A$ is singular in a fiber if and only if $4\Det(A)$ is 0 in that fiber. To form the matrix $A$ from the conic, we must use $1/2$, but
then $D=4\Det(A)$ is a polynomial with integer coefficients in the coefficients of the form defining the conic.  Even in characteristic 2, the polynomial $D$ gives
the exact condition for singularity.
This is the form that cuts out the singular locus of the pencil of conics through $A$ and $B$ (originally described in Deligne's letter \cite{Delquartic} to Bhargava).  
So when $4\Det(Ax-By)\in\Z[x,y]$ is not $0$, the cubic resolvent ring will be given by the regular functions on the subscheme of singular conics in the $\P^1_\Z$ of conics through $A$ and $B$.

\subsection{Geometric relationship between forms}\label{SS:geomZ}
We have a map $\P^1_\Z \ra \P^2_\Z$ given by anticanonical embedding of the projective line, or $[u:v]\mapsto [v^2:u^2:uv]$.
Note that $A_0$ gives a quadratic form on $\P^2_\Z$, and the scheme cut out by this form is the rational normal curve specified above.
If we have a pair $(A_0,B)$ (with $B$ not a multiple of $A_0$), then the conic given by $A_0$ in the pencil is not singular in any fiber. 
Thus, the associated cubic resolvent ring is given by the ring of regular functions of a closed subscheme of $\P^1_\Z\setminus\{A_0\}\isom \A^1_\Z$, and thus is monogenic.

We can see from the parametrization of cubic rings (Theorem~\ref{S:cubic}) that whenever a cubic ring is monogenic, in its realization as the global functions of a subscheme of $\P^1_\Z$, that subscheme actually
sits inside an $\A^1_\Z \subset \P^1_\Z$.
Thus, if the cubic resolvent ring associated to a nice (as above) pair $(A,B)$ is monogenic, then that means that the subscheme of singular conics in the 
$\P^1_\Z$ of conics through $A$ and $B$ is disjoint from some particular conic defined over $\Z$, and we can change basis of the pencil so that it is disjoint from $A$.
  This means that $A$ is non-singular.  From the fact that there is only one $\SL_3(\Z)$ class of ternary
quadratic forms with determinant $-1/4$, we know that up to $\GL_3(\Z)$ change of basis on $\P^2_\Z$, the only such conic is the one cut out by
$\pm A_0$.  So we see that pairs $(A_0,B)$ correspond to pairs of quartic rings and cubic resolvents such that the resolvents are monogenic.

Moreover, if we have a pair $(A_0,B)$, we can pull $B$ back to a form on the $\P^1_\Z$ cut out by $A_0$ to obtain a binary quartic form
(and we obtain the same binary quartic form with any element of $B+A_0\Z$).  
We can easily compute that every binary quartic form arises this way.  In particular, if $B=b_{11}x^2+b_{12}xy+b_{13}xz+b_{22}y^2+b_{23}yz+b_{33}z^2$,
we see it pulls back to the form $b_{11}v^4+b_{12}u^2v^2+b_{13}uv^3+b_{22}u^4+b_{23}u^3v+b_{33}u^2v^2$ (exactly inverse to the map $\Psi$ defined in the introduction).
The elements of $\GL_3(\Z)$ that fix the $\P^1_\Z$ cut out by $A_0$ setwise restrict to elements of $\GL_2(\Z)$ acting on that $\P^1_\Z$.
This allows us to see the correspondence of the $\GL_2(\Z)$ action on $\P^1_\Z$ and an action on $\P^2_\Z$ which fixes the rational normal curve.  

If we have a primitive binary quartic form $f$, then the scheme $S_f$ cut out by the form is $\Spec$ of the associated ring \cite[Theorem 2.9]{binarynic}.
The ideal class of $R_f$  associated to the form $f$ (as constructed in \cite{binarynic}) is the line bundle $\O(1)$ pulled back from $\P^1_\Z$ to $S_f$
and gives a map of $S_f$ into $\P^1_\Z$.
The scheme $S_f$ is also a subscheme of $\P^2_\Z$ cut out by $(A_0,B_f)$.  We can see the relationship here between the ideal and the monogenic cubic resolvent.
The ideal gives a map of $S_f$ to $\P^1_\Z$, and then by composing with the rational normal curve map into $\P^2_\Z$ we see from the above story that the cubic resolvent is monogenic.
Conversely, a monogenic cubic resolvent gives a smooth conic on which our degree four subscheme lies (as in the above story), and pulling back $\O(1)$ from this
conic (which is isomorphic to $\P^1_\Z$) gives the ideal associated to the binary quartic form.

\subsection{Analogs over an arbitrary base}\label{SS:AB}

In the construction of a quartic ring from a binary quartic form with $\Z$ coefficients, one can replace $\Z$ with an arbitrary scheme $S$.
An \emph{binary quartic form} over $S$ is a triple $(V,L,f)$, in which $V$ and $L$ are vector bundles on $S$ of ranks $2$ and $1$ respectively, and
$f\in H^0(S,\Sym^4U \tesnor L)$.  A\emph{double ternary quadratic form} over $S$ (analogous to a pair of ternary quadratic forms) is a 
quadruple $(W,U,p,\phi)$, in which $W$ and $U$ are vector bundles on $S$ of ranks $3$ and $2$ respectively, 
$p\in H^0(S,\Sym^2 W \tesnor U)$, and an isomorphism (called an orientation) $\wedge^3 W \tensor \wedge^2 U \stackrel{\sim}{\ra} \OS$.

From either a binary quartic form or double ternary quadratic forms over $S$ one can construct a quartic $\OS$-algebra (an $\OS$-algebra
which is a locally free rank 4 $\OS$-module), or equivalently, a degree 4 cover of $S$.  If we let $\pi$ denote the map $\P(V)\ra S$ (or, respectively,
$\P(W)\ra S$) and $K$ the Koszul complex of $f$ (respectively $p$), then the quartic algebra $R_f$ (respectively $R_p$) is given by $H^0 R\pi_* K$, with the algebra structure being inherited from
the multiplication on the Koszul complex.  More details on these constructions are given in \cite[Section 3]{binarynic} and \cite[Section 4]{Quartic}.  

Given a binary quartic form $f$ over $S$ which is a \emph{non-zero-divisor} (everywhere locally on $S$), then $f$ determines
a proper subscheme $S_f$ of $\P(V)$, of generic relative dimension $0$ over each irreducible component of $S$, and the associated quartic algebra $R_f$ is $\pi_* \O_{S_f}$ (see \cite[Section 3]{binarynic}).
Moreover, if for every $s\in S$, we have that $f$ is not identically zero in the fiber over $s$, then $S_f\ra S$ has relative dimension $0$ and is the degree 4 cover associated to $f$.
This is because $S_f\ra S$ is quasi-finite and projective, and thus finite and affine by \cite[Book 4, Chapter 18, \S 12]{EGAIV}.

The most geometric analogs of the work in this paper concern Gorenstein quartic algebras, and now we will see which binary quartic forms give Gorenstein quartic algebras.
\begin{proposition}\label{P:Gor}
 Given a binary quartic form $f$ over $S$, the following are equivalent:
\begin{enumerate}
 \item for every $s\in S$, we have that $f$ is not identically zero in the fiber over $s$
\item $\Spec R_f\ra S$ is Gorenstein.
\end{enumerate}
\end{proposition}
\begin{proof}
If for every $s\in S$, we have that $f$ is not identically zero in the fiber over $s$, then $\Spec R_f=S_f$, which is a complete intersection and thus Gorenstein. 
Conversely, suppose $\Spec R_f\ra S$ is Gorenstein but that $f$ is zero in the fiber over $s$.  Then, we have a simple computation that $R_f \tensor_{\OS} k(s)=k(s)[x,y,z]/(x,y,z)^2$
(c.f. Equation~\eqref{E:multtable}), which is not Gorenstein, a contradiction.
\end{proof}

The next proposition shows that for Gorenstein quartic covers, coming from a binary quartic form is equivalent to being a closed subscheme of a $\P^1$-bundle.
\begin{proposition}\label{P:e}
Given a quartic algebra $R$ over a scheme $S$, we have that the following are equivalent
\begin{enumerate}
 \item there exists a binary quartic form $f$ such that $R\isom R_f$ and $\Spec R\ra S$ is Gorenstein
\item $\Spec R \ra S$ factors through a $\P^1$ bundle $X\ra S$ such that $\Spec R \ra X$ is a closed immersion.
\end{enumerate} 
\end{proposition}
\begin{proof}
For an $\O_S$-vector bundle $W$, let $W^*$ denote the dual $\sHom(W,\O_S)$.
 The fact that (1) implies (2) is clear from the above.  If we have $\Spec R \sub \P(V)$ (as $S$ schemes), where $V$ is a vector bundle of rank $2$ on $S$,
then locally on $S$, we have that $\Spec R_f$ is cut out by a degree 4 equation, determined up to a unit on $S$.  Thus we have a line bundle
$L$ on $S$ and  $L\sub \pi_* \O(4)_{\P(V)}$
is the line bundle of sections vanishing on $\Spec R_f$.  This gives a map $L \ra \Sym^4 V$, or equivalently, a section of $H^0(\Sym^4 V\tesnor L^{*})$, which is a binary quartic form $f$.
Then from the above we see that $R$ is the quartic algebra associated to $f$, and that $\Spec R\ra S$ is Gorenstein, showing that (2) implies (1).
\end{proof}

One can also define cubic resolvents of quartic algebras over $S$ analogously to Definition~\ref{D:real} (see \cite[Section 1.1]{Quartic}).
For convenience, by way of the work of \cite{Quartic} that establishes a correspondence between double ternary quartic forms and quartic algebras with cubic resolvents analogous to Theorem~\ref{T:B},
we work as follows.  A double ternary quadratic form has a determinant binary cubic form $\Det(p)$ (see \cite[Section 3]{Quartic}), which locally agrees with the above construction
$4\Det(Ax-By)$.  If a double ternary quadratic form $p$ has an associated quartic algebra $R_p$, we call the cubic algebra associated to $\Det(p)$ (via the same construction as a 
quartic algebra from a binary quartic form) a \emph{cubic resolvent} of $R_p$.

Now we will see that a quartic algebra from a binary quartic form has a cubic resolvent that is a closed subscheme of (the total space of a) line bundle.  This is the analog over $S$ of monogenicity over $\Z$.

\begin{proposition}\label{P:m}
  If $R$ is a quartic algebra over $S$ satisfying the conditions of Proposition~\ref{P:e}, then  $R$ has a Gorenstein cubic resolvent $C$ such that $\Spec C$ is a closed sub-$S$-scheme of the total space of a line bundle  over $S$.  
\end{proposition}

\begin{proof}
Assuming (2) of Proposition~\ref{P:e}, we have $\Spec R \sub \P(V)$, and an associated binary quartic form $f\in H^0(\Sym^4 V\tesnor L^{*})$.  The idea of this proof is that we will produce a double ternary quadratic form that has associated quartic algebra $R$.  In this way, we will produce a cubic resolvent of $R$, and see that it has the desired property.
To produce a double ternary quadratic form that has associated quartic algebra $R$, we will embedd $R$ into a $\P^2$-bundle over $S$ and then use the pencil of conics vanishing on $R$ as our double ternary quadratic form.  One technical difficulty is that we must produce an orientation for our double ternary quadratic form.  We handle this by first constructing the most natural candidate for a double ternary quadratic form, measuring the line bundle that obstructs it from having an orientation, and then twisting things appropriately by that line bundle.

  The section $f$ naturally gives a map $g: L \ra \Sym^4 V$. 
Let $M \sub \Sym^2 \Sym^2 V$ be the pre-image of $\im(g)$ in the natural map  $\Sym^2 \Sym^2 V \ra \Sym^4 V$.  It is easy to see, by working locally, that $M$ is a rank $2$ vector bundle
over $S$.  Moreover, the forms in $M$ are the forms of $\Sym^2 \Sym^2 V$ that vanish on $\Spec R$ in the composite map $\Spec R \sub \P(V) \sub \P(\Sym^2 V)$.  So $M$ gives a pencil on conics
in $\P(\Sym^2 V)$ that cut out $\Spec R$ as a closed subscheme.  However, $M$ may not allow the orientation ismorphism required to be a double ternary quadratic form. 
We have a line bundle $P:=\wedge^3(\Sym^2 V)\tensor \wedge^2 M^*$, that measures the failure of $M$ to allow this orientation isomorphism.  

We now twist by $P$ to construct a double ternary quadratic form from the pencil of conics determined by $M$.
We have a map $h: L\tensor P^2 \ra \Sym^4 V\tensor P^2$ induced from the map $g$ above.
Let $N \sub \Sym^2 ((\Sym^2 V)\tesnor P)$ be the pre-image of $\im(h)$ in the natural map  $\Sym^2 ((\Sym^2 V)\tesnor P) \ra \Sym^4 V\tensor P^2$.
We have $N\isom M\tensor P^2$.

The map $N \sub \Sym^2 ((\Sym^2 V)\tesnor P)$ gives a section $p \in H^0(S,\Sym^2 ((\Sym^2 V)\tesnor P)\tesnor N^*)$ and thus a double ternary quadratic form $((\Sym^2 V)\tesnor P, N^*,p,\phi)$,
where $\phi$ is the natural map $\wedge^3 ((\Sym^2 V)\tesnor P ) \tensor \wedge^2 N^* \isom \wedge^3 (\Sym^2 V)\tesnor P^{\tensor 3} \tensor \wedge^2 M^* \tesnor P^{\tensor -4}\isom \O_S$.
Note that $\phi$ exists because our twist by $P$ was exactly the twist necessary to give the isomorphism $\phi$.

We have $\P(V) \sub \P(\Sym^2 V)\isom \P(\Sym^2 V \tensor P)$, cut out by the kernel of the map $\Sym^2 ((\Sym^2 V)\tesnor P) \ra \Sym^4 V\tensor P^2$.
By the same argument as in Section~\ref{SS:gc}, the cubic resolvent $C$ associated to $p$ is the cubic algebra associated to the binary cubic form $c$ cutting out the singular conics in the pencil of conics through
$\Spec R \sub \P(\Sym^2 V \tensor P)$.  Since $\P(V)$ is a smooth conic in this pencil, we see that $c$ is non-zero in each fiber over $S$, and thus the cubic resolvent is Gorenstein and $\Spec C$
is cut out by $c$ in the pencil of conics through $\Spec R$ (by the same argument as in Proposition~\ref{P:Gor}).  Moreover, since $\P(V)$ is a smooth conic, we see that $\Spec C$ is a closed subscheme not only in the $\P^1$ bundle of conics, but in the line bundle constructed by removing $\P(V)$ from the pencil.
\end{proof}

We have a partial converse to Proposition~\ref{P:m}.

\begin{proposition}\label{P:n}
   If $R$ is a Gorenstein quartic algebra over an integral scheme $S$, and $R$ has a Gorenstein cubic resolvent $C$ such that $\Spec C$ is a closed sub-$S$-scheme of the total space of a line bundle over $S$, then
  $\Spec R \ra S$ factors through a (smooth) conic bundle $X\ra S$ such that $\Spec R \ra X$ is a closed immersion.
\end{proposition}

\begin{proof}
 Let $(W,U,p,\phi)$ be a double ternary quadratic form giving $R$ and $C$.  
Since $R$ is Gorenstein, we can deduce that $p$ cuts out $\Spec R$ in $\P(W)$ (by a computation showing all forms over an algebraically closed field not cutting out a relative dimension $0$ $S$-scheme
give non-Gorenstein algebras).  Since $C\ra S$ is Gorenstein, $C$ (as an $S$-scheme) has a unique (up to isomorphism) closed immersion in a $\P^1$ bundle over $S$ by \cite[Theorem 1.3]{CasI}. 
Since $\Spec C$ is a closed sub-$S$-scheme of the total space of a line bundle over $S$, completing that line bundle to a $\P^1$-bundle gives such a closed immersion.

 Also, $\Spec C$ is the singular locus in the pencil of conics through $\Spec R$, which gives a closed immersion of $\Spec C$ into a $\P^1$-bundle (the pencil of conics).  Thus in this closed immersion $\Spec C$ must sit inside a line bundle in the $\P^1$ bundle.  The difference between the line bundle and the $\P^1$-bundle gives a section of $S$ into the $\P^1$-bundle (the pencil of conics).  This section corresponds to a conic bundle over $S$ through $\Spec R$, but since it has no intersection with $\Spec C$, it is a smooth conic bundle.
\end{proof}

Over a general scheme $S$, a conic bundle is not necessary a $\P^1$-bundle, and thus Proposition~\ref{P:n} does not give a full converse to Proposition~\ref{P:m}.  Over $\Z$, of course, there is a unique (smooth) conic bundle, which is a $\P^1$-bundle and we have a full converse.  Also, the explicit argument given for Theorem~\ref{T:Main} works for all quartic algebras, whereas the geometric arguments given above are restricted to Gorenstein quartic algebras.  (When one tries to extend to the non-Gorestein locus by the hypercohomology construction of the quartic algebras, one recovers essentially the explicit argument given in the first part of this paper.)  The recent results on class groups of monogenic cubic rings \cite{BS} are not restricted by any Gorenstein condition, and thus over $\Z$ it is convenient to have the complete statement of Theorem~\ref{T:Main}.

\section{$\GL_2(\Z)$ invariants of binary quartic forms and cubic resolvent rings}\label{S:Canonical}

We have a canonical map $\Psi'$
which sends $f=f_0 x^4 +f_1 x^3y+f_2x^2y^2+f_3 xy^3 +f_4y^4$ to
$$
\left( A_0
,
 \begin{pmatrix} f_4 & \frac{f_2}{6} & \frac{f_3}{2} \\
\frac{f_2}{6} & f_0 & \frac{f_1}{2} \\
\frac{f_3}{2} & \frac{f_1}{2} &  \frac{2f_2}{3}
\end{pmatrix}
\right),
$$
which is equivariant with respect the $\GL_2(\Z)$ action on the binary quartic forms and the $\GL_2(\Z)$ action
on pairs of ternary quadratic forms given in Theorem~\ref{T:action}. Our previous map $\Psi$  was equivariant only as a map to
$N$ classes of pairs of ternary quadratic forms.  However $\Psi$ was defined over $\Z$, and $\Psi'$ requires the use of $\frac{1}{3}$.   

The determinant binary cubic of any element in the image of $\Psi'$ has $x^3y$ coefficient 0.
If $3\mid f_2$, then $\Psi'(f)$ has integral coefficients and its determinant is the unique binary cubic form to give a basis $1,\omega,\theta$ such that $\omega\theta\in \Z$ and 
$\omega^2-\theta\in\Z$,
where $\omega$ is the generator of the resolvent cubic associated to the form $f$.
If $3\nmid f_2$, then there is no basis $1,\omega,\theta$ of the monogenized resolvent cubic $C,\omega$ such that $\omega\theta\in \Z$ and  $\omega^2-\theta\in\Z$. 

We  define $N_{1/3}$ to be the group of matrices of the form $\left(\begin{smallmatrix}
                                 1 & 0\\
				n & 1
                                \end{smallmatrix}\right)$,
where $n\in\frac{1}{3}\Z$.

\begin{proposition}
The map from $N$ classes of monogenic binary cubic forms to $N_{1/3}$ classes of binary cubic forms
is injective.
\end{proposition}

\begin{proof}
Consider the action of $\left(\begin{smallmatrix}
                                 1 & \\
				k/3 & 1
                                \end{smallmatrix}\right)$ (where $k\in\Z$) on the form
 $-x^3+bx^2y+cxy^2+dy^3$.  The new coefficient of $y^3$ is
 $d-\frac{ck}{3}+\frac{bk^2}{9}-\frac{k^2}{27}$, which is only an integer if $k$ is divisible by 3.  
\end{proof}

The determinant of $\Psi'(f)$ is
 $$-x^3+\frac{I}{3}xy^2-\frac{J}{27}y^3,$$
where $I$ and $J$ are generators for the $\SL_2(\Z)$ invariants of binary quartic forms, given by
$$\frac{I}{3}=4f_0f_4-f_1f_3+\frac{1}{3} f_2^2 \quad \text{and} \quad \frac{-J}{27}=\frac{-8}{3}f_0f_2f_4+\frac{2}{27}f_2^3+
f_0f_3^2+f_4f_1^2-\frac{1}{3}f_1f_2f_3.
$$
Given a binary quartic form, we have an associated quartic ring and a monogenized cubic resolvent $C$ with generator $\omega$.  We can thus give the monogenized cubic resolvent canonically by saying it corresponds to the $N$ class of 
 binary cubic forms over $\Z$ in the 
$N_{1/3}$ class of  $-x^3+\frac{I}{3}xy^2-\frac{J}{27}y^3.$

Let $r$ be a root of $-x^3+\frac{I}{3}x-\frac{J}{27}$.  Then there is only one $\Z$ coset of algebraic integers in $r +\frac{1}{3}\Z$, and it is
$\omega+\Z$.  So, we have found a description for $\omega$ in terms of the $\GL_2(\Z)$ invariants of the binary quartic form.  
(Note that even if $-x^3+\frac{I}{3}x-\frac{J}{27}$ is reducible, we can still make sense of $r$ as an element of 
$\Q(r)/(-r^3+\frac{I}{3}r-\frac{J}{27})$ and there is only one $\Z$ coset in $r +\frac{1}{3}\Z$ whose elements generate algebras
that are finitely generated $\Z$-modules.)

\section{Appendix: Cubic Resolvents}\label{A:CubRes}
Let $Q$ be a quartic ring. The following definition is an important step in constructing the cubic resolvent of a general quartic ring, and was first given in \cite[Definition 6]{HCL3}, though see 
\cite{Sk} for a more thorough treatment.

\begin{definition}
 The \emph{$S_4$-closure of $Q$}, denoted $\bar{Q}$, is the ring $Q^{\tensor 4}/J_Q$, where $J_Q$ is the $\Z$-saturation of the ideal $I_Q$ generated
by all the elements of the form $$x\tesnor1\tensor1\tensor1 + 1\tesnor x\tensor 1\tensor1 +1\tesnor 1\tensor x\tensor1+1\tesnor1\tensor1\tensor x$$ for $x\in Q$
(that is, $J_Q=\{r\in Q^{\tensor 4} | nr\in I_Q \textrm{ for some } n\in\Z\}$).
\end{definition}

The idea is that $1\tesnor x\tensor 1\tensor1,1\tesnor 1\tensor x\tensor1,1\tesnor1\tensor1\tensor x$ act as formal conjugates of $x\tesnor1\tensor1\tensor1$.
We then consider the cubic invariant ring of $Q$, that is
$$
R^{inv}(Q):=\Z[\{x\tesnor x\tensor1\tensor1 + 1\tesnor 1\tensor x\tensor x | x\in Q\}]\sub \bar{Q}.
$$

Now we can give the definition of a cubic resolvent ring as given in \cite[Definition 8]{HCL3}.
\begin{definition}
 If $Q$ is a quartic ring with non-zero discriminant, \emph{a cubic resolvent ring} $C$ of $Q$ is any cubic ring with $R^{inv}(Q)$ as a subring such that $\disc(C)=\disc(Q)$.
Then $x\mapsto x\tesnor x\tensor1\tensor1 + 1\tesnor 1\tensor x\tensor x$ gives a quadratic map from $Q$ to $C$.
\end{definition}

We also give the following alternative definition (\cite[Definition 20]{HCL3}) which works for all quartic rings, and is easier to use in practice in some cases.

\begin{definition}\label{D:real}
 Given a quartic ring $Q$, \emph{a cubic resolvent}
$C$ of $Q$ is
\begin{itemize}
 \item a cubic ring $C$ 
 \item a quadratic map $\phi : Q/\Z \ra C/\Z$, and
 \item an isomorphism $\delta : \wedge^4 Q \isom \wedge^3 C$
(or equivalently $\bar{\delta}: \wedge^3 Q/\Z \isom \wedge^2 C/\Z$)
\end{itemize}
such that
\begin{enumerate}
 \item for all $x,y\in Q$, 
we have $\delta(1\wedge x\wedge y\wedge xy)=1 \wedge \phi(x)\wedge \phi(y) $
 \item the binary cubic form in $\Sym^3(C/\Z)^* \tensor \wedge^2(C/Z)$ corresponding to $C$ is $\Det(\phi)$.
\end{enumerate}
\end{definition}

\section*{Acknowledgements}
The author would like to thank Manjul Bhargava for asking the questions that inspired this research, guidance along the way, and helpful feedback both on the ideas
and the exposition in this paper. 
The author would also like to thank the referee for several suggestions that improved the paper.
Part of this work was done while the author was supported by an NSF Graduate Fellowship, an NDSEG Fellowship, an AAUW Dissertation Fellowship, and a Josephine De K\'{a}rm\'{a}n Fellowship.  Part of this work was done while the author was supported by an American Institute of Mathematics Five-Year Fellowship and National Science Foundation grant DMS-1001083.


\begin{thebibliography}{99} 


\bibitem{HCL3} M. Bhargava, Higher composition laws. III. The parametrization of quartic rings, Ann. of Math. (2) {\bf 159} (2004), no.~3, 1329--1360.

\bibitem{Sk} M. Bhargava and M. Satriano, On a notion of ``Galois closure'' for extensions of rings, preprint.

\bibitem{BS} M. Bhargava and A. Shankar, ``Binary quartic forms having bounded invariants, and the boundedness of the average rank of elltipc curves,'' preprint (arXiv:1006.1002v2).

\bibitem{BM} 
B. J. Birch\ and\ J. R. Merriman, Finiteness theorems for binary forms with given discriminant, Proc. London Math. Soc. (3) {\bf 24} (1972), 385--394.

\bibitem{BosmaStevenhagen} W. Bosma\ and\ P. Stevenhagen, On the computation of quadratic $2$-class groups, J. Th\'eor. Nombres Bordeaux {\bf 8} (1996), no.~2, 283--313.

\bibitem{CasI}
G. Casnati\ and\ T. Ekedahl, Covers of algebraic varieties. I. A general structure theorem, covers of degree $3,4$ and Enriques surfaces, J. Algebraic Geom. {\bf 5} (1996), no.~3, 439--460.


\bibitem{Cassels} J. W. S. Cassels, {\it Rational quadratic forms}, Academic Press, London, 1978.

\bibitem{Cohen} H. Cohen, {\it A course in computational algebraic number theory}, Springer, Berlin, 1993.

\bibitem{DD}
R. Dedekind, Supplement X in P.G.L. Dirichlet, {\it Vorlesungen \"{u}ber Zahlentheorie,} Vieweg, second edition (1871).


\bibitem{PrimeSplit}
I. Del Corso, R. Dvornicich\ and\ D. Simon, Decomposition of primes in non-maximal orders, Acta Arith. {\bf 120} (2005), no.~3, 231--244.

\bibitem{Delcubic} P. Deligne, letter to  W. T. Gan, B. Gross\ and\ G. Savin, November 13, 2000.

\bibitem{Delquartic} P. Deligne, letter to  M. Bhargava, March 5, 2004.

\bibitem{DF}B. N. Delone\ and\ D. K. Faddeev, {\it The theory of irrationalities of the third degree}, Amer. Math. Soc., Providence, R.I., 1964. (translation of B. N. Delone\ and\ D. K. Faddeev, Theory of Irrationalities of Third Degree, Acad. Sci. URSS. Trav. Inst. Math. Stekloff, {\bf 11} (1940).)

\bibitem{Fla} D. E. Flath, {\it Introduction to number theory}, Wiley, New York, 1989. 

\bibitem{GGS} W. T. Gan, B. Gross\ and\ G. Savin, Fourier coefficients of modular forms on $G\sb 2$, Duke Math. J. {\bf 115} (2002), no.~1, 105--169.

\bibitem{EGAIV}A. Grothendieck, \'{E}l\'ements de g\'eom\'etrie alg\'ebrique. {IV}. \'{E}tude
              locale des sch\'emas et des morphismes de sch\'emas {IV}, Inst. Hautes \'Etudes Sci. Publ. Math. No. 32 (1967), 361 pp.


\bibitem{Kneser} Kneser, Martin. Composition of binary quadratic forms.  J. Number Theory  15  (1982), no. 3, 406--413.


\bibitem{Naka}
J. Nakagawa, Binary forms and orders of algebraic number fields, Invent. Math. {\bf 97} (1989), no.~2, 219--235. 

\bibitem{Shaf} I. R. Shafarevich, {\it Basic algebraic geometry. 1}, Second edition, Translated from the 1988 Russian edition and with notes by Miles Reid, Springer, Berlin, 1994. 

\bibitem{Simon}
D. Simon, The index of nonmonic polynomials, Indag. Math. (N.S.) {\bf 12} (2001), no.~4, 505--517.

\bibitem{SimonIdeal}
D. Simon, La classe invariante d'une forme binaire, C. R. Math. Acad. Sci. Paris {\bf 336} (2003), no.~1, 7--10.

\bibitem{Simonobst}
D. Simon, A ``class group'' obstruction for the equation $Cy\sp d=F(x,z)$, J. Th\'eor. Nombres Bordeaux {\bf 20} (2008), no.~3, 811--828.


\bibitem{binarynic} M. M. Wood, Rings and ideal parametrized by binary $n$-ic forms, J. London Math. Soc. (2) 83 (2011) 208–-231.

\bibitem{BinQuad} M. M. Wood, Gauss composition over an arbitrary base, Adv. Math., 226 (2011) 1756-1771. 

\bibitem{Quartic} M. M. Wood, Parametrizing quartic algebras over an arbitrary base, to appear in Algebra and Number Theory, arXiv:1007.5503v1.

\end{thebibliography}
\end{document}